\documentclass[11pt,a4paper,reqno]{amsart}
\usepackage{stmaryrd}
\usepackage[english]{babel}
\usepackage{pst-grad} % For gradients
\usepackage{pst-plot} % For axes
\usepackage{pstricks}
\usepackage{amsmath,amssymb,mathrsfs,amsthm, mathtools}
\usepackage{tikz} 
\usepackage{mathcomp,wasysym}  
\usepackage{graphicx}  
\usepackage[all]{xy} \xyoption{arc} \xyoption{color}
\usepackage{epsfig}
\usepackage{cite}
\usepackage[a4paper,
left=2.5cm, right=2.5cm,
top=3cm, bottom=3cm]{geometry}
\newcommand{\pare}[1]{\left( #1 \right)}
\newcommand{\norm}[1]{\left\| #1 \right\|}

\newcommand{\bra}[1]{\left[ #1 \right]}

\newcommand{\sgn}{\textnormal{sgn}}

\newcommand{\cH}{\mathcal{H}}

\newcommand{\bR}{\mathbb{R}}

\newcommand{\cO}{\mathcal{O}}

\normalsize
\normalsize
\def\comm#1#2{{\left\llbracket#1,#2\right\rrbracket}}

\usepackage[bookmarks,colorlinks]{hyperref}

\newtheorem{satz}{Proposition}[section]
\newtheorem{remark}[satz]{Remark} 
\newtheorem{theorem}{Theorem}

\newtheorem{lemma}[satz]{Lemma}

\title{On the motion of gravity-capillary waves with odd viscosity}

\author[R. Granero-Belinch\'{o}n]{Rafael Granero-Belinch\'{o}n}
\email{rafael.granero@unican.es}
\address{Departamento  de  Matem\'aticas,  Estad\'istica  y  Computaci\'on,  Universidad  de Cantabria.  Avda.  Los  Castros  s/n,  Santander,  Spain.}

\author[A. Ortega]{Alejandro Ortega}
\email{alortega@math.uc3m.es}
\address{Dpto. de Matem\'aticas, Universidad Carlos III de Madrid, Av. de la Universidad 30, 28911 Legan\'es (Madrid),
Spain}

\begin{document}
\begin{abstract}
We develop three asymptotic models of surface waves in a non-newtonian fluid with odd viscosity. This viscosity is also known as Hall viscosity and appears in a number of applications such as quantum Hall fluids or chiral active fluids. Besides the odd viscosity effects, these models capture both gravity and capillary forces up to quadratic interactions and take the form of nonlinear and nonlocal wave equations. Two of these models describe bidirectional waves while the third PDE studies the case of unidirectional propagation. We also prove the well-posedness of these asymptotic models in spaces of analytic functions and in Sobolev spaces. Finally, we present a number of numerical simulations for the unidirectional model.
\end{abstract}


%\ccode{2010 AMS Subject Classification: 35455, 35B41, 92C17}

\maketitle
{\small
\tableofcontents}

\allowdisplaybreaks
\section{Introduction}
The equations describing the motion of an incompressible fluid take the following form
\begin{align*}
\rho\left(\frac{\partial}{\partial t}u+(u\cdot\nabla)u\right)&=\nabla\cdot \mathscr{T} &&\text{ in }\Omega(t)\times[0,T],\\
\nabla\cdot u&=0&&\text{ in }\Omega(t)\times[0,T],\\
\frac{\partial}{\partial t}\rho+\nabla\cdot(u \rho)&=0 &&\text{ in }\Omega(t)\times[0,T],
\end{align*}
where $u, \rho$ and $\mathscr{T}$ denote the velocity, density and stress tensor of the fluid respectively. This stress tensor takes different forms depending on the physical properties of the fluid. On the one hand, we have the case of inviscid fluids where the previous equations reduce to the well-known Euler system with a stress tensor given by
$$
\mathscr{T}^i_j=-p\delta^i_j.
$$
On the other hand, we have the case of viscous newtonian fluids where the tensor takes the classical form
$$
\mathscr{T}^i_j=-p\delta^i_j+\nu_e\left(\frac{\partial }{\partial x_j}u^i+\frac{\partial }{\partial x_i}u^j\right).
$$
Due to the symmetries of the tensor, this viscosity is called \emph{even} or \emph{shear} viscosity.

Fluids in which both time reversal and parity are broken can display a dissipationless
viscosity that is odd under each of these symmetries. This viscosity is called \emph{odd} or Hall viscosity.  In this case the stress tensor takes the following form (cf. \cite{khain2020stokes})
$$
\mathscr{T}^i_j=-p\delta^i_j+\nu_o\left(\frac{\partial }{\partial x_i}(u^j)^\perp+\left(\frac{\partial }{\partial x_i}\right)^\perp u^j\right),
$$
where, for a vector $a=(a_1,a_2)$ we used the notation
$$
a^\perp=(a_2,-a_1).
$$

Situations where such a viscosity arise in a natural way are, for instance, the motion of quantum Hall fluids at low temperature (cf. \cite{avron1995viscosity}), the motion of vortices (cf. \cite{wiegmann2014anomalous}) or chiral active fluids (cf. \cite{banerjee2017odd}) among other applications (see \cite{souslov2019topological} or \cite{abanov2020hydrodynamics} and the references therein).

Since the seminar works of Avron, Seiler \& Zograf \cite{avron1995viscosity} (see also \cite{avron1998odd,banerjee2017odd,ganeshan2017odd,ganeshannon,lapa2014swimming,souslov2019topological}) the effect of the odd viscosity has been an active research area. Fluids that experience odd or Hall viscosity have a rather counterintuitive behavior. For instance, while a rotating disk immersed in a  fluid with even viscosity experiments a friction force that oposses the rotation, the same disk rotating in a fluid with odd viscosity feels a pressure in the radial direction (see \cite[Section 5.3]{avron1998odd} and \cite{lapa2014swimming} for more details).

Whilst in three dimensions, terms in the viscosity tensor with odd symmetry were known in the context of anisotropic fluids, Avron noticed that in two dimensions odd viscosity and isotropy can hold at the same time (cf. \cite{avron1998odd}). This motivated the study of two dimensional incompressible flows with odd viscosity effects. This situation can be described by the following system (cf. \cite{avron1998odd,ganeshan2017odd})

\begin{subequations}\label{eq:oddNS1}
\begin{align}
\rho\left(\frac{\partial}{\partial t}u+(u\cdot\nabla)u\right)&=-\nabla p+\nu_o\Delta u^\perp &&\text{ in }\Omega(t)\times[0,T],\\
\nabla\cdot u&=0&&\text{ in }\Omega(t)\times[0,T],\\
\frac{\partial}{\partial t}\rho+\nabla\cdot(u\rho)&=0 &&\text{ in }\Omega(t)\times[0,T],
\end{align}
\end{subequations}
where $u, \rho$ and $p$ denote the velocity, density and pressure of the fluid and $\nu_o$ is a positive constant reflecting the influence of odd viscosity. 

\section{The free boundary problem}
\subsection{Derivation}
We observe that using the divergence free condition, system \eqref{eq:oddNS1} can be rewritten as
\begin{subequations}\label{eq:oddNS}
\begin{align}
\rho\left(\frac{\partial}{\partial t}u+(u\cdot\nabla)u\right)&=-\nabla (p-\nu_o (\nabla\cdot u^\perp)) &&\text{ in }\Omega(t)\times[0,T],\\
\nabla\cdot u&=0&&\text{ in }\Omega(t)\times[0,T],\\
\frac{\partial}{\partial t}\rho+u\cdot\nabla\rho&=0 &&\text{ in }\Omega(t)\times[0,T].
\end{align}
\end{subequations}
This system needs to be supplemented with appropriate initial and boundary conditions.

We stress that the viscosity tensor takes the form of a gradient and, as a consequence, it can be absorbed into the pressure in the bulk of the fluid once we define the modified pressure
$$
\tilde{p}=p-\nu_o (\nabla\cdot u^\perp).
$$
However, the odd viscosity affects the boundary conditions (cf. \cite{ganeshan2017odd}). This is of particular importance when considering free boundary problems for flows under the effect of odd viscosity. Such free boundary flows have been reported in experiments by  Soni, Bililign, Magkiriadou, Sacanna, Bartolo, Shelley \& Irvine (see \cite{soni2018free,soni2019odd}). The purpose of this paper is to study the dynamics of a incompressible, irrotational and homogeneous fluid bounded above by a free interface under the effect of gravity, surface tension and, more importantly, odd viscosity.

We consider the (to be determined) moving domain
$$
\Omega(t)=\left\{(x_1,x_2)\in \mathbb{R}^2, x_1\in\bR, x_2<\eta(x_1,t)\right\},
$$
for certain function $\eta$. The free interface is then
$$
\Gamma(t)=\left\{(x_1,x_2)\in \mathbb{R}^2, x_1\in\bR, x_2=\eta(x_1,t)\right\}.
$$
This function $\eta$ is advected by the fluid and as a consequence it satisfies a transport equation. 

Since we consider surface tension effects, the stress tensor satisfies the following boundary condition
$$
\mathscr{T}^i_jn_j=\gamma \mathcal{K}n_i,
$$
where $\vec{n}$ is the unit upward pointing normal to the surface wave, $\gamma$ is the surface tension strength and $\mathcal{K}$ denotes the curvature of the free boundary
$$
\mathcal{K}=\frac{\eta_{x_1x_1}}{\left(1+\eta_{x_1}^2\right)^{3/2}}\quad \text{ on }\Gamma(t).
$$

Since the fluid is homogeneous, we have that the density $\rho$ remains constant. It is well-known that the occurrence of viscosity creates a boundary layer near the surface wave. However, the thickness of this boundary layer is small when the physical parameters are in a certain regime (see \cite{dias2008theory,lamb1932hydrodynamics} for instance). In particular, the vorticity is confined to this narrow boundary layer near the surface (see \cite{abanov2018odd}). Then, following the ideas in \cite{abanov2018odd,dias2008theory} and the references therein, we consider potential flow with a modified boundary conditions for the redefined pressure $\tilde{p}$. Then, we assume that
$$
u=\nabla \theta,
$$
where $\theta$ is the scalar velocity potential. This velocity potential, by virtue of the incompressibility condition, is harmonic. Then, the effects of boundary layer are captured by the odd viscosity modified pressure term (cf.\cite{abanov2018odd,ganeshannon})
$$
\tilde{p}=-\gamma \mathcal{K}-\frac{2\nu_{o}}{(1+\eta_{x_1}^2)^{1/2}}\left(u\cdot n\right)_{x_1}.
$$
This new boundary condition for the modified pressure encodes the contribution of the narrow boundary layer and gives and accurate description of the problem.

Thus, to describe surface waves under the effect of gravity, capillary forces and odd viscosity we have to consider the following free boundary problem \cite{abanov2018odd,abanov2019free,ganeshannon}
\begin{subequations}\label{eq:wavesodd}
\begin{align}
\Delta \theta&=0&&\text{ in }\Omega(t)\times[0,T],\\
\rho\left(\theta_t+\frac{\theta_{x_1}^2+\theta_{x_2}^2}{2}+G\eta\right)-\gamma\mathcal{K}&=\frac{2\nu_{o}}{(1+\eta_{x_1}^2)^{1/2}}\left(\frac{\eta_t}{(1+\eta_{x_1}^2)^{1/2}}\right)_{x_1}&&\text{ on }\Gamma(t)\times[0,T],\\
\eta_t&=-\eta_{x_1}\theta_{x_1}+\theta_{x_2}&&\text{ on }\Gamma(t)\times[0,T],
\end{align}
\end{subequations}
where $\theta$ is the scalar potential (units of $length^2/time$), $\eta$ denotes the surface wave (units of $length$) and $G$ (units of $length/time^2$) is the gravity acceleration. The constant $\nu_o$ reflects the odd viscosity contribution and has units of $(length^2/ time)$. 

As is customary (see \cite{zakharov1968stability}), we use the trace of the velocity potential 
$$
\xi(t,x_1,x_2)=\theta(t,x_1,\eta(t,x_1)).
$$
Thus, \eqref{eq:wavesodd} can be written as
\begin{subequations}\label{eq:all2}
\begin{align}
\Delta \theta&=0&&\text{ in }\Omega(t)\times[0,T],\\
\theta  &= \xi \qquad &&\text{ on }\Gamma(t)\times[0,T],\\
\xi_t+\frac{\theta_{x_1}^2+\theta_{x_2}^2}{2}+G\eta&=\frac{\gamma}{\rho}\frac{\eta_{x_1x_1}}{\left(1+\eta_{x_1}^2\right)^{3/2}}+\theta_{x_2}\left(-\eta_{x_1}\theta_{x_1}+\theta_{x_2}\right)&&\nonumber\\
&\quad+\frac{2\nu_{o}}{\rho}\frac{1}{(1+\eta_{x_1}^2)^{1/2}}\left(\frac{\eta_t}{(1+\eta_{x_1}^2)^{1/2}}\right)_{x_1}&&\text{ on }\Gamma(t)\times[0,T],\\
\eta_t&=-\eta_{x_1}\theta_{x_1}+\theta_{x_2}&&\text{ on }\Gamma(t)\times[0,T].
\end{align}
\end{subequations}
The system (\ref{eq:all2}) is supplemented with an initial condition for $h$ and $\xi$:
\begin{align}\label{eq:initial}
h(x,0)&=h_0(x),\\
\xi(x,0)&=\xi_0(x).
\end{align}
\subsection{Prior works}
System \eqref{eq:all2} is the odd viscosity analogue to the water waves system with (even) viscosity in the work of Dias, Dyachenko \& Zakharov (see \cite{dias2008theory}). A careful inspection shows that the system with odd viscosity is of purely dispersive nature while the system in \cite{dias2008theory} is of cross-diffusion type. This free boundary problem with even viscosity has received a large amount of attention in recent years both from the applied mathematics viewpoint, which is interested in new and better mathematical models that capture the main dynamics in suitable regimes, and from the pure mathematics community, that studies the dynamics of the underlying differential equations. We refer the reader to  \cite{dutykh2007viscous,dutykh2009visco,dutykh2007dissipative,kakleas2010numerical,
ngom2018well,ambrose2012well,GS19,granero2020well,granero2020well2}. 

Even if the odd viscosity effects are most visible at the free surface of the fluids (see \cite{abanov2018odd,ganeshan2017odd,abanov2020hydrodynamics}) and the phenomenon has been described experimentally (cf. \cite{soni2018free,soni2019odd}), the number of results studying surface waves with odd viscosity remains, to the best of our knowledge, small. 

Let us briefly summarize the available literature on surface waves with odd viscosity. Very recently, Abanov, Can \& Ganeshan \cite{abanov2018odd} considered the free surface dynamics of a two-dimensional incompressible fluid with odd viscosity. Besides studying the dispersion relation of such waves derived a number of weakly nonlinear models. First, after neglecting gravity, surface tension and terms of cubic order, they obtained the following Craig-Sulem-type model (see equations (49) and (50) in \cite{abanov2018odd})
\begin{align*}
u_t-\left[\cH(u\cH u)\right]_{x_1}+2\nu_o \Lambda u_{x_1}&=-2\nu_o\left[\comm{\cH}{h}\cH u\right]_{x_1 x_1 x_1},\\
h_t+\cH u&=-\left[\comm{\cH}{h}\cH u\right]_{x_1},
\end{align*}
where $\comm{\cdot}{\cdot}$ denotes the commutator and $\cH$ and $\Lambda$ stand for the Hilbert transform and the fractional Laplacian respectively (see below for a proper definition). As the authors point out in their work, this system is Hamiltonian. In addition, the authors also considered the small surface angle approximation to conclude the following model (see equations (51) and (52) in \cite{abanov2018odd})
\begin{align*}
h_{tt}&=-\left[\cH h_t h_t\right]_{x_1}-2\nu_o  \cH h_{tx_1 x_1}.
\end{align*}
This latter equation has the same nonlinearity as the $h-$model in \cite{granero2017model} while keeping a linear operator akin to the classical Benjamin-Ono equation. Indeed, the previous model can be equivalently written as
\begin{align*}
v_{t}&=-\left[\cH v v\right]_{x_1}-2\nu_o  \cH v_{x_1 x_1}.
\end{align*}
In this new variable we recognize a well-known nonlinearity already heavily studied in the literature (see \cite{castro2008global,bae2015global,li2011one} and the references therein). 

In addition, Abanov, Can \& Ganeshan \cite{abanov2018odd}, starting from the previous small angle approximation, derived and studied what they named the \emph{chiral} Burgers equation
$$
u_t+2uu_{x_1}=2i\nu_o u_{x_1 x_1}.
$$ 
We would like to emphasize that all these models are obtained with \emph{heuristic} arguments instead of a more rigorous asymptotic approximation.

Later on, Monteiro \& Abanov \cite{abanov2019free} presented a variational principle which accounts for odd viscosity effects in free boundary incompressible flows.

Moreover, Monteiro \& Ganeshan \cite{ganeshannon} studied the case of waves with odd viscosity in a shallow fluid and derived the celebrated KdV equation as a model in the long wavelength weakly nonlinear regime.

\subsection{Contributions and main results}
The purpose of this paper is twofold. On the one hand, we obtain three new models for capillary-gravity surface waves with odd viscosity. These new models are obtained through a multiscale expansion in the steepness of the wave and extend the previous results in \cite{aurther2019rigorous,GS19,granero2019asymptotic}. Furthermore, our models consider both gravity and surface tension forces and, as a consequence, generalize those in \cite{abanov2018odd}. In particular, we obtain the model
\begin{align}\label{eq:final_eq_intro}
f_{tt}&=-\Lambda f-\beta\Lambda^3f+\alpha_o \Lambda f_{tx_1}+\varepsilon\left[-\cH\left((\cH f_t)^2\right)+\left(\comm{\cH}{f}\Lambda f\right)\right]_{x_1}\nonumber\\
&\quad +\varepsilon\left[-\alpha_o\left(\comm{\cH}{f}\Lambda f_{tx_1}\right)+\beta\comm{\cH}{f}\Lambda^3 f\right]_{x_1}&&\text{ on }\Gamma\times[0,T].
\end{align}
Noticing that the terms that are $\cO(\varepsilon \alpha_o)$ and $\cO(\varepsilon \beta)$ are much smaller that the rest, we can also consider the following PDE
\begin{align}\label{eq:final_eq2_intro}
f_{tt}&=-\Lambda f-\beta\Lambda^3f+\alpha_o \Lambda f_{tx_1}+\varepsilon\left[-\cH\left((\cH f_t)^2\right)+\left(\comm{\cH}{f}\Lambda f\right)\right]_{x_1}\text{ on }\Gamma\times[0,T].
\end{align}
Similarly, if we restrict ourselves to the study of unidirectional surface waves, we can derive the following dispersive equation
\begin{equation}\label{eq:final_eq3_intro}
\begin{split}
2f_{t}+\alpha_o\Lambda f_{t}=&\frac{1}{\varepsilon}\left\{f_{x_1}+\cH f+(\alpha_o-\beta)\cH f_{x_1x_1}\right\}\\
&+\cH\left(\Lambda f\right)^2-\comm{\cH}{f}\Lambda f+(\alpha_o-\beta)\comm{\cH}{f}\Lambda^3 f,\ \ \text{ on }\Gamma\times[0,T].
\end{split}
\end{equation}
On the other hand, we prove a number of mathematical results establishing the well-posedness of our new models in appropriate functional spaces  (see \cite{granero2018asymptotic,granero2020well2,granero2020global} for some related results).

Roughly speaking, we prove the following theorems (see below for the precise statements)
\begin{enumerate}
\item Equation \eqref{eq:final_eq_intro} is locally well-posed for analytic initial data. The proof is based on a Cauchy-Kovalevski type argument that relies in finding appropriate bounds for a cascade of linear equations as in \cite{aurther2019rigorous}.
\item Equation \eqref{eq:final_eq2_intro} is locally well-posed in $H^{4.5}(\bR)\times H^{3}(\bR)$ when the Bond number $\beta>0$. The required energy estimates exploit the commutator structure of the nonlinearity in a very precise way.
\item Equation \eqref{eq:final_eq3_intro} is locally well-posed in $H^{3}(\bR)$ when the odd Reynolds number $\alpha_o$ is strictly positive regardless of the value of the Bond number $\beta$. Furthermore, when $0<\alpha_o=\beta$, the problem admits a distributional solution in $H^{1.5}(\bR)$.
\end{enumerate}

The plan of the paper is as follows. First, in section \ref{ref:sec2} we write the dimensionless problem of gravity-capillary waves with odd viscosity in the Arbitrary Lagrangian-Eulerian formulation. Next, in section \ref{ref:sec3} we derive and study the case of a bidirectional gravity-capillary wave with odd viscosity. In particular, we obtain two new nonlinear and nonlocal wave equations (see equations \eqref{eq:final_eq_intro} and \eqref{eq:final_eq2_intro}). These PDEs describe the main dynamics in the weakly nonlinear regime and consider the case where the steepness parameter $\varepsilon$, the odd Reynolds number $\alpha_o$ and Bond number $\beta$ are small. Moreover, we establish the local strong well-posedness of \eqref{eq:final_eq_intro} and \eqref{eq:final_eq2_intro} in appropriate functional spaces. Then, in section \ref{ref:sec4} we study the case of unidirectional waves and obtain the new nonlocal and nonlinear dispersive equation \eqref{eq:final_eq3_intro} in the case of right-moving waves. Furthermore, we establish the local strong well-posedness of \eqref{eq:final_eq3_intro} in Sobolev spaces provided that the odd Reynolds number is positive. Finally, we also prove a local in time existence of distributional solution with limited regularity for \eqref{eq:final_eq3_intro}. In section \ref{ref:sec5} we conclude with a brief discussion presenting the main novelties of our work.

\subsection{Notation}
Given a matrix $A$, we write $A^i_j$ for the component of $A$, located on row $i$ and column $j$. We will use the Einstein summation convention for expressions with indexes.

We write
$$
f_{x_j}=\frac{\partial f}{\partial x_j},\quad f_t=\frac{\partial f}{\partial t }
$$
for the space derivative in the $j-$th direction and for a time derivative, respectively. 

Let $f(x_1)$ denote a $L^2$ function on $\mathbb{R}$. We define the
Hilbert transform $\mathcal{H}$ and the Dirichlet-to-Neumann operator $ \Lambda $ and its powers,  respectively,  using Fourier series
\begin{align}\label{Hilbert}
\widehat{\mathcal{H}f}(k)=-i\text{sgn}(k) \hat{f}(k) \,, \ \ 
\widehat{\Lambda f}(k)=|k|\hat{f}(k)\,,
\end{align}
where 
$$
\hat f(k) = {\dfrac{1}{\sqrt{2\pi}}} \int_{\mathbb{R}} f(x_1) \ 
e^{-ikx_1}dx_1.
$$
Finally, given an operator $\mathcal{T}$, we define the commutator as
$$
\comm{\mathcal{T}}{f}g=\mathcal{T}(fg)-f\mathcal{T}(g).  
$$

\section{Gravity-capillary waves with odd viscosity}\label{ref:sec2}
As mentioned above, we use the trace of the velocity potential 
$$
\xi(t,x_1,x_2)=\theta(t,x_1,\eta(t,x_1)).
$$
Thus, \eqref{eq:wavesodd} can be written in dimensionless form (see \cite{granero2017model} for more details) as
\begin{subequations}\label{eq:alldimensionless}
\begin{align}
\Delta \theta&=0&&\text{ in }\Omega(t)\times[0,T],\\
\theta  &= \xi \qquad &&\text{ on }\Gamma(t)\times[0,T],\\
\xi_t+\varepsilon\frac{\theta_{x_1}^2+\theta_{x_2}^2}{2}+\eta&=\frac{\beta\eta_{x_1x_1}}{\left(1+\left(\varepsilon\eta_{x_1}\right)^2\right)^{3/2}}+\varepsilon\theta_{x_2}\left(-\varepsilon\eta_{x_1}\theta_{x_1}+\theta_{x_2}\right)&&\\
&\quad+\alpha_o\frac{1}{(1+\varepsilon^2\eta_{x_1}^2)^{1/2}}\left(\frac{\eta_t}{(1+\varepsilon^2\eta_{x_1}^2)^{1/2}}\right)_{x_1}&&\text{ on }\Gamma(t)\times[0,T],\\
\eta_t&=-\varepsilon\eta_{x_1}\theta_{x_1}+\theta_{x_2}&&\text{ on }\Gamma(t)\times[0,T],
\end{align}
\end{subequations}
where $\varepsilon$ is known as the \emph{steepness parameter} and measures the ratio between the amplitude and the wavelength of the wave while $\alpha_o$ is a dimensionless parameter akin to the Reynolds number that represents the ratio between gravity and odd viscosity forces. Due to this similarity, we call it \emph{odd Reynolds number}. Similarly, $\beta$ is the Bond number comparing the gravity and capillary forces.

Now we want to express system \eqref{eq:alldimensionless} on the reference domain $\Omega$ and reference interface $\Gamma$
\begin{align}\label{Omega}
\Omega = \mathbb{R} \times (-\infty, 0) \,, &&
\Gamma = \mathbb{R} \times \{0\} \,.
\end{align} 
In order to do that we define the following family of diffeomorphisms
\begin{equation*}
\begin{aligned}
&\psi : && \bra{0, T} \times\Omega && \to && \Omega\pare{t}, \\
&&& \pare{x_1, x_2, t} && \mapsto && \psi\pare{x_1, x_2,t} = \pare{x_1, x_2 + \varepsilon \eta\pare{x_1,t}}. 
\end{aligned}
\end{equation*}
We compute 
\begin{align}\label{eq:diif_matrices}
\nabla \psi = \pare{
\begin{array}{cc}
1 & 0 \\
\varepsilon \eta_{x_1}\pare{x_1,t} & 1
\end{array}
}, &&
A = \pare{\nabla\psi}^{-1} = \pare{
\begin{array}{cc}
1 & 0 \\
-\varepsilon \eta_{x_1}\pare{x_1,t} & 1
\end{array}
}. 
\end{align}
We define the ALE variables 
$$
\Theta = \theta \circ \psi.
$$
We can equivalently write \eqref{eq:alldimensionless} in the (fixed) reference domain as follows
\begin{subequations}\label{eq:ALE}
\begin{align}
A^\ell_j\left(A^k_j\Theta_{x_{k}}\right)_{x_\ell}&=0&&\text{ in }\Omega\times[0,T],\\
\Theta  &= \xi \qquad &&\text{on }\Gamma\times[0,T],\\
\xi_t+\frac{\varepsilon}{2}A^k_j\Theta_{x_k} A^\ell_j\Theta_{x_\ell}+\eta&=\frac{\beta\eta_{x_1x_1}}{\left(1+\left(\varepsilon\eta_{x_1}\right)^2\right)^{3/2}}+\varepsilon A^k_2\Theta_{x_k} A^\ell_j\Theta_{x_\ell} A^2_j&&\\
&\quad+\alpha_o\frac{1}{(1+\varepsilon^2\eta_{x_1}^2)^{1/2}}\left(\frac{\eta_t}{(1+\varepsilon^2\eta_{x_1}^2)^{1/2}}\right)_{x_1}&&\text{ on }\Gamma\times[0,T],\\
\eta_t&=A^k_j\Theta_{x_k} A^2_j&&\text{ on }\Gamma\times[0,T].
\end{align}
\end{subequations}

Using the explicit value of $A^k_j$, we can regroup terms and find that
\begin{subequations}\label{eq:ALE2}
\begin{align}
\Delta \Theta &={\varepsilon\pare{\eta_{x_1x_1} \ \Theta_{x_2} + 2\eta_{x_1} \Theta_{x_1x_2}}-\varepsilon^2(\eta_{x_1})^2\Theta_{x_2x_2}}  ,  &&\text{in } \Omega\times[0,T]\,,\\
\Theta  &= \xi \qquad &&\text{on }\Gamma\times[0,T],\\
\xi_t &=-\frac{\varepsilon}{2}\left[(\Theta_{x_1})^2+(\varepsilon \eta_{x_1}\Theta_{x_2})^2+(\Theta_{x_2})^2-2\varepsilon \eta_{x_1}\Theta_{x_2}\Theta_{x_1}\right]&&\nonumber\\
&\quad- \eta +\varepsilon \Theta_{x_2}\left(-\varepsilon \eta_{x_1} \Theta_{x_1}+\varepsilon^2 (\eta_{x_1})^2 \Theta_{x_2} +\Theta_{x_2}\right)\nonumber\\
&\quad+\frac{\beta\eta_{x_1x_1}}{\left(1+\left(\varepsilon\eta_{x_1}\right)^2\right)^{3/2}}\nonumber\\
&\quad+\alpha_o\frac{1}{(1+\varepsilon^2\eta_{x_1}^2)^{1/2}}\left(\frac{\eta_t}{(1+\varepsilon^2\eta_{x_1}^2)^{1/2}}\right)_{x_1}&&\text{ on }\Gamma\times[0,T],\\
\eta_t&=-\varepsilon \eta_{x_1} \Theta_{x_1}+\varepsilon^2 (\eta_{x_1})^2 \Theta_{x_2} +\Theta_{x_2}&&\text{ on }\Gamma\times[0,T].
\end{align}
\end{subequations}

\section{The bidirectional asymptotic model for waves with odd viscosity} \label{ref:sec3}
\subsection{Derivation}
We are interested in a model approximating the dynamics up to $\mathcal{O}(\varepsilon^2)$. As a consequence a number of terms can be neglected with this order of approximation and we find

\begin{subequations}\label{eq:ALE22}
\begin{align}
\Delta \Theta &={\varepsilon\pare{\eta_{x_1x_1} \ \Theta_{x_2} + 2\eta_{x_1} \Theta_{x_1x_2}}}  ,  &&\text{in } \Omega\times[0,T]\,,\\
\Theta  &= \xi \qquad &&\text{on }\Gamma\times[0,T],\\
\xi_t &=-\frac{\varepsilon}{2}\left[(\Theta_{x_1})^2+(\Theta_{x_2})^2\right]- \eta +\varepsilon \Theta_{x_2}^2+\beta \eta_{x_1x_1}+\alpha_o\eta_{tx_1}&&\text{ on }\Gamma\times[0,T],\\
\eta_t&=-\varepsilon \eta_{x_1} \Theta_{x_1} +\Theta_{x_2}&&\text{ on }\Gamma\times[0,T].
\end{align}
\end{subequations}

In order to obtain the asymptotic model for the interface under the effect of odd viscosity we will assume the following form for the unknowns
\begin{equation}
\label{eq:ansatz}
\begin{aligned}
\Theta\pare{x_1,x_2,t} & = \sum_n \varepsilon^n \Theta^{\pare{n}}\pare{x_1,x_2,t}, \\
\xi\pare{x_1,t} & = \sum_n \varepsilon^n \xi^{\pare{n}}\pare{x_1,t}, \\
\eta\pare{x_1,t} & = \sum_n \varepsilon^n \eta^{\pare{n}}\pare{x_1,t}.
\end{aligned}
\end{equation}

For the case $n=0$, we have that
\begin{subequations}\label{eq:n0}
\begin{align}
\Delta \Theta^{(0)} &=0 ,  &&\text{in } \Omega\times[0,T]\,,\\
\Theta^{(0)}  &= \xi^{(0)}  \qquad &&\text{on }\Gamma\times[0,T],\\
\xi^{(0)} _t &=- \eta^{\pare{0}}+\beta \eta^{(0)}_{x_1x_1} +\alpha_o\eta^{\pare{0}}_{tx_1}  &&\text{ on }\Gamma\times[0,T],\\
\eta_t^{(0)} &=\Theta^{(0)}_{x_2} &&\text{ on }\Gamma\times[0,T].
\end{align}
\end{subequations}
The solution of the associated elliptic problem for the first term of the velocity potential is given by
$$
\widehat{\Theta^{\pare{0}}} \pare{k, x_2, t}  = \xi^{(0)}(k,t)e^{|k|x_2}\qquad \text{in } \Omega\times[0,T],
$$
so
$$
\Theta^{(0)}_{x_2}=\Lambda \xi^{(0)}\qquad \text{on }\Gamma.
$$
Then, we find that the linear problem for the first term of the series for the interface is
\begin{align*}
\eta_{tt}^{(0)} &=-\Lambda\eta^{\pare{0}}-\beta\Lambda^3\eta^{\pare{0}}+\alpha_o\Lambda\eta^{\pare{0}}_{tx_1} &&\text{ on }\Gamma\times[0,T].
\end{align*}

The second term in the expansion is 
\begin{subequations}\label{eq:ALE2v2}
\begin{align}
\Delta \Theta^{(1)} &=\eta^{(0)}_{x_1x_1} \ \Theta^{(0)}_{x_2} + 2\eta^{(0)}_{x_1} \Theta^{(0)}_{x_1x_2}  ,  &&\text{in } \Omega\times[0,T]\,,\\
\Theta^{(1)}  &= \xi^{(1)} \qquad &&\text{on }\Gamma\times[0,T],\\
\xi_t^{(1)} &=-\frac{1}{2}\left[(\Theta^{(0)}_{x_1})^2+(\Theta^{(0)}_{x_2})^2\right]- \eta^{(1)} + (\Theta^{(0)}_{x_2})^2+\beta\eta^{\pare{1}}_{x_1x_1}+\alpha_o\eta_{tx_1}^{(1)}&&\text{ on }\Gamma\times[0,T],\\
\eta^{(1)}_t&=-\eta^{(0)}_{x_1} \Theta^{(0)}_{x_1} +\Theta^{(1)}_{x_2}&&\text{ on }\Gamma\times[0,T].
\end{align}
\end{subequations}

We recall now the following lemma
\begin{lemma}[\cite{GS19}]\label{lem:solutions_Poisson}
Let us consider the Poisson equation
\begin{equation}
\label{eq:Poisson}
\left\lbrace
\begin{aligned}
& \Delta u \pare{x_1 , x_2} &&= b \pare{x_1 , x_2}, & \pare{x_1, x_2} &\in \mathbb{R}\times  \pare{-\infty, 0}, \\
%-----------------------------------------------------------
&  u \pare{x_1, 0} &&= g \pare{x_1} , & x_1 &\in \mathbb{R}, \\
& \lim_{x_2\rightarrow -\infty}\partial_2 u \pare{x_1, x_2} && =0, & x_1 &\in \mathbb{R},
\end{aligned}
\right. 
\end{equation}
where we  assume that the forcing $ b$ and the boundary data $g$ are smooth and decay fast enough at infinity. Then, 
\begin{align}
u_{x_2} \pare{x_1, 0} & =\int _{ -\infty}^0  e^{ y_2 \Lambda}b\pare{x_1, y_2} \textnormal{d} y_2+\Lambda g(x_1)\label{eq:pa2u0}.
\end{align}
\end{lemma}

Using this lemma, we find that
\begin{equation}\label{eq:explicit_expression_Phi1}
\left.  \Theta^{\pare{1}}_{x_2}\right|_{x_2=0} =   \Lambda \xi^{\pare{1}} -\comm{\Lambda}{\eta^{\pare{0}}}\Lambda \xi^{ \pare{0}}
\end{equation}
and we can write system \eqref{eq:ALE2v2} as
\begin{subequations}\label{eq:ALE2v22}
\begin{align}
\xi_t^{(1)} &=-\frac{1}{2}\left[(\xi^{(0)}_{x_1})^2+(\Lambda\xi^{(0)})^2\right]- \eta^{(1)} + (\Lambda\xi^{(0)})^2+\beta\eta_{x_1x_1}^{(1)}+\alpha_o\eta_{tx_1}^{(1)}&&\text{ on }\Gamma\times[0,T],\\
\eta^{(1)}_t&=-\eta^{(0)}_{x_1} \xi^{(0)}_{x_1} +\Lambda \xi^{\pare{1}} -\comm{\Lambda}{\eta^{\pare{0}}}\Lambda \xi^{ \pare{0}}&&\text{ on }\Gamma\times[0,T].
\end{align}
\end{subequations}

Recalling that
$$
\eta^{(0)}_t=\Lambda \xi^{(0)},
$$
and Tricomi's identity
\begin{equation}\label{tricomi}
(\cH f)^2-f^2=2\cH\left(f\cH f\right),
\end{equation}
we can compute
\begin{align*}
\xi_t^{(1)} &=\beta\eta_{x_1x_1}^{(1)}-\cH\left(\eta^{(0)}_t\cH \eta^{(0)}_t\right)- \eta^{(1)} +\alpha_o\eta_{tx_1}^{(1)}&&\text{ on }\Gamma\times[0,T],\\
\eta^{(1)}_t&=\eta^{(0)}_{x_1} \cH \eta^{(0)}_t +\Lambda \xi^{\pare{1}} -\comm{\Lambda}{\eta^{\pare{0}}}\eta^{(0)}_t&&\text{ on }\Gamma\times[0,T].
\end{align*}
Taking a time derivative and substituting the value of $\Lambda \xi^{(1)}_t$, we compute that
\begin{align*}
\eta^{(1)}_{tt}&=-\Lambda\eta^{(1)}-\beta\Lambda^3\eta^{(1)}+\alpha_o \Lambda \eta^{(1)}_{tx_1}-\Lambda\left((\cH \eta_t^{(1)})^2\right)+\left(\comm{\cH}{\eta^{(0)}}\Lambda \eta^{(0)}\right)_{x_1}\\
&\quad+\beta\left(\comm{\cH}{\eta^{(0)}}\Lambda^3 \eta^{\pare{0}}\right)_{x_1}-\alpha_o\left(\comm{\cH}{\eta^{(0)}}\Lambda \eta^{(0)}_{tx_1}\right)_{x_1}&&\text{ on }\Gamma\times[0,T].
\end{align*}

If we now define
$$
f(x_1,t)=\eta^{(0)}(x_1,t)+\varepsilon\eta^{(1)}(x_1,t),
$$
after neglect terms of order $\cO (\varepsilon^2)$, we conclude the following bidirectional model of gravity-capillary waves with odd viscosity

\begin{align}\label{eq:final_eq}
f_{tt}&=-\Lambda f-\beta\Lambda^3f+\alpha_o \Lambda f_{tx_1}+\varepsilon\left[-\cH\left((\cH f_t)^2\right)+\left(\comm{\cH}{f}\Lambda f\right)\right]_{x_1}\nonumber\\
&\quad +\varepsilon\left[-\alpha_o\left(\comm{\cH}{f}\Lambda f_{tx_1}\right)+\beta\comm{\cH}{f}\Lambda^3 f\right]_{x_1}&&\text{ on }\Gamma\times[0,T].
\end{align}

With an appropriate choice of the parameters, \eqref{eq:final_eq} recovers the quadratic $h-$model in \cite{aurther2019rigorous, matsuno1992nonlinear, matsuno1993two, matsuno1993nonlinear, AkMi2010, AkNi2010}.

Furthermore, we observe that some of the terms are $\cO(\varepsilon \alpha_o)$ and $\cO(\varepsilon \beta)$, \emph{i.e.} they are much smaller than the rest of the nonlinear contributions. Then one can expect that they can be neglected to find 
\begin{align}\label{eq:final_eq2}
f_{tt}&=-\Lambda f-\beta\Lambda^3f+\alpha_o \Lambda f_{tx_1}+\varepsilon\left[-\cH\left((\cH f_t)^2\right)+\left(\comm{\cH}{f}\Lambda f\right)\right]_{x_1}\text{ on }\Gamma\times[0,T].
\end{align}
A similar equation was obtained in \cite{GS19,granero2020well2,granero2020global} for the case of damped waves under the effect of even viscosity.

\subsection{Well-posedness for analytic initial data}
We recall the definition of the Wiener spaces in the real line (see \cite{gancedo2020surface} for more properties)
$$
\mathbb{A}_\tau(\bR)=\left\{h\in L^1(\bR) \text{ s.t. }\|h\|_{\mathbb{A}_\tau}=\int_{\bR}e^{\tau|n|}|\hat{h}(n)|dn<\infty\right\}.
$$

This section is devoted to the proof of the following result

\begin{theorem}\label{teo1}Let $\beta,\alpha_o\geq0$ be fixed constants. Assume that the initial data for equation \eqref{eq:final_eq} satisfies
$$
(f_0,f_1)\in L^1(\mathbb{R})\times L^1(\mathbb{R})
$$
and
\begin{equation*}
\begin{split}
f_0(x_1)&=\frac{1}{\sqrt{2\pi}}\int_{-D}^{D}\hat{f}_0(k)e^{ikx_1}dk,\\
f_1(x_1)&=\frac{1}{\sqrt{2\pi}}\int_{-D}^{D}\hat{f}_1(k)e^{ikx_1}dk,
\end{split}
\end{equation*}
for some $1<D<+\infty$. Then there exists $0<T^*$ and a solution to equation \eqref{eq:final_eq} such that
$$
(f,f_t)\in L^\infty(0,T^*;\mathbb{A}_{1}(\bR))\times L^\infty(0,T^*;\mathbb{A}_{1}(\bR))\cap C(0,T^*;\mathbb{A}_{0.5}(\bR))\times C(0,T^*;\mathbb{A}_{0.5}(\bR)).
$$
\end{theorem}
\begin{proof}
We look for a solution of the form 
\begin{equation}\label{eq:sol}
f(x,t)=\sum_{\ell=0}^{\infty}\lambda^{\ell+1}f^{(\ell)}(x,t)
\end{equation}
for some $\lambda$ to be fixed later. By substituting this expression into \eqref{eq:final_eq} and matching terms we get that $f^{(\ell)}$ satisfies the equation
\begin{equation}\label{eq:1}
\begin{split}
f_{tt}^{(\ell)}=&-\Lambda f^{(\ell)}-\beta\Lambda^3f^{(\ell)}+\alpha_o\Lambda f_{tx_1}^{(\ell)}\\
             &+\sum_{j=0}^{\ell-1}\Big[-\cH\left(\cH f_t^{(j)}\cH f_t^{(\ell-1-j)}\right)+\cH\left(f^{(j)}\Lambda f^{(\ell-1-j)}\right)-f^{(j)}\cH\Lambda f^{(\ell-1-j)}\Big]_{x_1}\\
						 &-\alpha_o\sum_{j=0}^{\ell-1}\Big[\left(\cH\left(f^{(j)}\Lambda f_{tx_1}^{(\ell-1-j)}\right)-f^{(j)}\cH\Lambda f_{tx_1}^{(\ell-1-j)}\right)\Big]_{x_1}\\
						 &+\beta\sum_{j=0}^{\ell-1}\Big[\cH\left(f^{(j)}\Lambda^3f^{(\ell-1-j)}\right)-f^{(j)}\cH\Lambda^3f^{(\ell-1-j)}\Big]_{x_1}
\end{split}
\end{equation}
with initial conditions 
\begin{equation*}
f^{(l)}(x_1,0)=\left\{
        \begin{tabular}{cr}
        $0$  &if  $\ell\neq 0$, \\
				&\\
        $\frac{f_0}{\lambda}$  &if $\ell=0$.
        \end{tabular}
        \right.
				\quad\text{and}\quad
f_t^{(l)}(x_1,0)=\left\{
        \begin{tabular}{cr}
        $0$  &if  $\ell\neq 0$, \\
				&\\
        $\frac{f_1}{\lambda}$  &if $\ell=0$.
        \end{tabular}
        \right.
\end{equation*}
Using the Fourier series expansion, from \eqref{eq:1} we get that each $\widehat{f}^{(\ell)}(x_1,t)$ satisfies the differential equation
\begin{equation*}
\widehat{f}_{tt}^{(\ell)}(k,t)=-|k|\widehat{f}^{(\ell)}(k,t)-\beta|k|^3\widehat{f}^{(\ell)}(k,t)+i\,\alpha_ok|k|\widehat{f}_{t}^{(\ell)}(k,t)+F(k,t)
\end{equation*}
where 
\begin{equation}\label{eq:2}
\begin{split}
F(k,t)=&|k|\sum_{j=0}^{\ell-1}\int_{-\infty}^{\infty}\sgn(m)\widehat{f}_t^{(j)}(m,t)\sgn(k-m)\widehat{f}_t^{(\ell-1-j)}(k-m,t)dm\\
%&\mkern+40mu +\int_{-\infty}^{\infty} \widehat{f}^{(j)}(m,t)|k-m|\widehat{f}^{(\ell-1-j)}(k-m,t)dm\\
%&+k\sum_{j=0}^{\ell-1}\int_{-\infty}^{\infty}\widehat{f}^{(j)}(m,t)\sgn(k-m)|k-m|\widehat{f}^{(\ell-1-j)}(k-m,t)dm\\
&+\sum_{j=0}^{\ell-1}\int_{-\infty}^{\infty}\widehat{f}^{(j)}(m,t)\widehat{f}^{(\ell-1-j)}(k-m,t)\Big[|k||k-m|-k(k-m)\Big]dm\\
%&+i\alpha_o|k|\sum_{j=0}^{\ell-1}\int_{-\infty}^{\infty}\widehat{f}^{(j)}(m,t)(k-m)|k-m|\widehat{f}_t^{(\ell-1-j)}(k-m,t)dm\\
%&-i\alpha_ok\sum_{j=0}^{\ell-1}\int_{-\infty}^{\infty}\widehat{f}^{(j)}(m,t)(k-m)^2\widehat{f}_t^{(\ell-1-j)}(k-m,t)dm\\
&+i\alpha_o\sum_{j=0}^{\ell-1}\int_{-\infty}^{\infty}\widehat{f}^{(j)}(m,t)\widehat{f}_t^{(\ell-1-j)}(k-m,t)(k-m)\Big[|k||k-m|-k(k-m)\Big]dm\\
%&+\beta |k|\sum_{j=0}^{\ell-1}\int_{-\infty}^{\infty}\widehat{f}^{(j)}(m,t)|k-m|^3\widehat{f}^{(\ell-1-j)}(k-m,t)dm\\
%&-\beta k\sum_{j=0}^{\ell-1}\int_{-\infty}^{\infty}\widehat{f}^{(j)}(m,t)\sgn(k-m)|k-m|^3\widehat{f}^{(\ell-1-j)}(k-m,t)dm\\
&+\beta \sum_{j=0}^{\ell-1}\int_{-\infty}^{\infty}\widehat{f}^{(j)}(m,t)\widehat{f}^{(\ell-1-j)}(k-m,t)(k-m)^2\Big[|k||k-m|-k(k-m)\Big]dm.
\end{split}
\end{equation}
Solving \eqref{eq:1} for $\ell=0$ we get 
\begin{equation*}
\widehat{f}^{(0)}(k,t)=\frac{1}{i(r^+-r^-)}\left\{[\hat{f}_1-ir^-\hat{f}_0]e^{ir^+t}-[\hat{f}_1-ir^+\hat{f}_0]e^{ir^-t}\right\}
\end{equation*}
where
\begin{equation*}
r^{\pm}=r^{\pm}(k)=\frac{\alpha_ok|k|\pm\sqrt{\alpha_o^2k^4+4(|k|+\beta|k|^3)}}{2}.
\end{equation*}
%which can be written as
%\begin{equation*}
%\widehat{f}^{(0)}(k,t)=e^{im^-t}\hat{f}_0-2i\sin\left(\frac{m^+-m^-}{2}t\right)\frac{e^{i\frac{\alpha_ok|k|}{2}t}}{m^+-m^-}[m^-\hat{f}_0+i(\hat{f}_0)_t]
%\end{equation*}
Similarly, for $\ell>0$ we obtain 
\begin{equation*}
\widehat{f}^{(\ell)}(k,t)=\frac{1}{i(r^+-r^-)}\int_{0}^{t}F(k,s)\left\{e^{ir^+(t-s)}-e^{ir^-(t-s)}\right\}ds
\end{equation*}
with $F(k,t)$ given in \eqref{eq:2}. Thus
\begin{equation*}
\widehat{f}_t^{(\ell)}(k,t)=\frac{1}{(r^+-r^-)}\int_{0}^{t}F(k,s)\left\{r^+e^{ir^+(t-s)}-r^-e^{ir^-(t-s)}\right\}ds.
\end{equation*}
Let us exploit the commutator structure of the nonlinearity (cf. \cite{granero2018asymptotic}). In particular, we note that
\begin{equation*}
[|k||k-m|-k(k-m)]=|k||k-m|[1-\sgn(k)\sgn(k-m)]\leq 2|k||k-m|
\end{equation*}
for $k<m$ and 
$$
[|k||k-m|-k(k-m)]=0
$$ otherwise. Hence 
\begin{equation}\label{eq:3}
\begin{split}
\frac{1}{r^+-r^-}[|k||k-m|-k(k-m)]&\leq2\sqrt{|m|}|k-m|\\
&\leq2(1+|m|)|k-m|.
\end{split}
\end{equation}

Let us fix $1<D< R\in\mathbb{Z}_+$ such that 
\begin{equation*}
\frac{D(R+1)}{1+R^2}\leq1.
\end{equation*}
Since the series \eqref{eq:sol} are respectively bounded by 
\begin{equation}\label{eq:solbound}
\sup\limits_{0\leq t\leq T}\sum_{l=0}^{\infty}\lambda^{\ell+1}\|f^{(\ell)}(t)\|_{\mathbb{A}_1}\quad\text{and}\quad \sup\limits_{0\leq t\leq T}\sum_{l=0}^{\infty}\lambda^{\ell+1}\|f_t^{(\ell)}(t)\|_{\mathbb{A}_1}
\end{equation}
by proving the boundedness of \eqref{eq:solbound} we get the absolute convergence of \eqref{eq:sol} and, hence, the existence of solutions. We start by considering the truncated series 
\begin{equation}\label{eq:truncseries}
S_R^1\vcentcolon=\sum_{\ell=0}^{R}\lambda^{\ell+1}f^{(\ell)}(x_1,t)\quad\text{and}\quad S_R^2\vcentcolon=\sum_{\ell=0}^{R}\lambda^{\ell+1}f_t^{(\ell)}(x_1,t).
\end{equation}
Thus, because of \eqref{eq:2} and \eqref{eq:3}, we obtain 
\begin{equation*}
\begin{split}
\|f^{(\ell)}(t)\|_{\mathbb{A}_{R+1-l}}\leq&\int_{-\infty}^{\infty}e^{(R+1-\ell)|k|}\int_{0}^{t}\sum_{j=0}^{\ell-1}\left[\int_{-\infty}^{\infty}\sqrt{|k|}|\hat{f}_t^{(j)}(m,s)||\hat{f}_t^{(\ell-1-j)}(k-m,s)|dm\right.\\
&+2\int_{-\infty}^{\infty}|k-m|(1+|m|)|\hat{f}^{(j)}(m,s)||\hat{f}^{(\ell-1-j)}(k-m,s)|dm\\
&+2\alpha_o\int_{-\infty}^{\infty}|k-m|^2(1+|m|)|\hat{f}^{(j)}(m,s)||\hat{f}_t^{(\ell-1-j)}(k-m,s)|dm\\
&+2\beta\left.\int_{-\infty}^{\infty}|k-m|^3(1+|m|)|\hat{f}^{(j)}(m,s)||\hat{f}^{(\ell-1-j)}(k-m,s)|dm\right]dsdk\\
%\leq&\int_{0}^{t}\sum_{j=0}^{\ell-1}\left[\|f_t^{(j)}(s)\|_{\mathbb{A}_{R+2-\ell}}\|f_t^{(\ell-1-j)}(s)\|_{\mathbb{A}_{R+2-l}}\right.\\
%&\mkern+50mu+4\|f^{(j)}(s)\|_{X_{R+2-\ell}}\|f^{(\ell-1-j)}(s)\|_{\mathbb{A}_{R+2-\ell}}\\
%&\mkern+50mu+8\alpha_o\|f^{(j)}(s)\|_{\mathbb{A}_{R+2-l}}\|f_t^{(\ell-1-j)}(s)\|_{X_{R+2-\ell}}\\
%&\mkern+50mu+24\beta\left.\|f^{(j)}(s)\|_{\mathbb{A}_{R+2-\ell}}\|f^{(\ell-1-j)}(s)\|_{\mathbb{A}_{R+2-\ell}}\right]ds\\
\leq&C_1(\alpha_o,\beta)\int_{0}^{t}\sum_{j=0}^{\ell-1}\left[\|f_t^{(j)}(s)\|_{\mathbb{A}_{R+2-\ell}}\|f_t^{(\ell-1-j)}(s)\|_{\mathbb{A}_{R+2-\ell}}\right.\\
&\mkern+120mu+\|f^{(j)}(s)\|_{\mathbb{A}_{R+2-\ell}}\|f^{(\ell-1-j)}(s)\|_{\mathbb{A}_{R+2-\ell}}\\
&\mkern+120mu+\|f^{(j)}(s)\|_{\mathbb{A}_{R+2-\ell}}\|f_t^{(\ell-1-j)}(s)\|_{\mathbb{A}_{R+2-\ell}}\\
&\mkern+120mu+\left.\|f^{(j)}(s)\|_{\mathbb{A}_{R+2-\ell}}\|f^{(\ell-1-j)}(s)\|_{\mathbb{A}_{R+2-\ell}}\right]ds
\end{split}
\end{equation*}
where we have used Fubini's theorem together with $\hat{f}^{(l)}(0,t)=0$ and the inequalities
\begin{equation*}
\begin{split}
1+|k|&\leq e^{|k|}\ \forall k\in\mathbb{R}\\
|k|^n&\leq n! e^{|k|}\ \forall k\in\mathbb{R}\\
|k|&\leq c e^{\frac{|k|}{c}}\leq ce^{\frac{|k-m|+|m|}{c}}\ \forall c\in\mathbb{Z}_+.
\end{split}
\end{equation*}
Moreover, using the commutator structure again, we obtain that
\begin{equation*}
\begin{split}
\left|\frac{r^{\pm}}{r^+-r^-}\right|[|k||k-m|-k(k-m)]
&\leq2 (|m|+\alpha_o|m|^2+\alpha_o|m|^3)|k-m|
\end{split}
\end{equation*}
for $k<m$ and, hence, we also find that
\begin{equation*}
\begin{split}
\|f_t^{(\ell)}(t)\|_{\mathbb{A}_{R+1-\ell}}\leq&C_2(\alpha_o,\beta)\int_{0}^{t}\sum_{j=0}^{\ell-1}\left[\|f_t^{(j)}(s)\|_{\mathbb{A}_{R+2-\ell}}\|f_t^{(\ell-1-j)}(s)\|_{\mathbb{A}_{R+2-\ell}}\right.\\
&\mkern+120mu+\|f^{(j)}(s)\|_{\mathbb{A}_{R+2-\ell}}\|f^{(\ell-1-j)}(s)\|_{\mathbb{A}_{R+2-\ell}}\\
&\mkern+120mu+\|f^{(j)}(s)\|_{\mathbb{A}_{R+2-\ell}}\|f_t^{(\ell-1-j)}(s)\|_{\mathbb{A}_{R+2-\ell}}\\
&\mkern+120mu+\left.\|f^{(j)}(s)\|_{\mathbb{A}_{R+2-\ell}}\|f^{(\ell-1-j)}(s)\|_{\mathbb{A}_{R+2-\ell}}\right]ds.
\end{split}
\end{equation*}
On the other hand, since $R+2-\ell\leq R+2-\ell+j=R+1-(\ell-1-j)$, we have
\begin{equation*}
\|f^{(\ell-1-j)}(s)\|_{\mathbb{A}_{R+2-\ell}}\leq\|f^{(\ell-1-j)}(s)\|_{\mathbb{A}_{R+1-(\ell-1-j)}}.
\end{equation*}
Furthermore, since $R+2-\ell=R+1-(\ell-1)\leq R+1-j$ for $j\leq \ell-1$ we also have
\begin{equation*}
\|f^{(j)}(s)\|_{\mathbb{A}_{R+2-l}}\leq\|f^{(j)}(s)\|_{\mathbb{A}_{R+1-j}}.
\end{equation*}
As a consequence, we find
\begin{equation*}
\begin{split}
\|f^{(\ell)}(t)\|_{\mathbb{A}_{R+1-\ell}}+\|f_t^{(\ell)}(t)\|_{\mathbb{A}_{R+1-\ell}}\leq&C(\alpha_o,\beta)\int_{0}^{t}\sum_{j=0}^{\ell-1}\left[\|f_t^{(j)}(s)\|_{\mathbb{A}_{R+2-\ell}}\|f_t^{(\ell-1-j)}(s)\|_{\mathbb{A}_{R+2-\ell}}\right.\\
%&\mkern+120mu+\|f^{(j)}(s)\|_{\mathbb{A}_{R+2-\ell}}\|f^{(\ell-1-j)}(s)\|_{\mathbb{A}_{R+2-\ell}}\\
&\mkern+120mu+\|f^{(j)}(s)\|_{\mathbb{A}_{R+2-\ell}}\|f_t^{(\ell-1-j)}(s)\|_{\mathbb{A}_{R+2-\ell}}\\
&\mkern+120mu+\left.\|f^{(j)}(s)\|_{\mathbb{A}_{R+2-\ell}}\|f^{(\ell-1-j)}(s)\|_{\mathbb{A}_{R+2-\ell}}\right]ds\\
\leq&C(\alpha_o,\beta)\int_{0}^{t}\sum_{j=0}^{\ell-1}\left[\|f_t^{(j)}(s)\|_{\mathbb{A}_{R+1-j}}\|f_t^{(\ell-1-j)}(s)\|_{\mathbb{A}_{R+1-(\ell-1-j)}}\right.\\
%&\mkern+120mu+\|f^{(j)}(s)\|_{\mathbb{A}_{R+1-j}}\|f^{(\ell-1-j)}(s)\|_{\mathbb{A}_{R+1-(\ell-1-j)}}\\
&\mkern+120mu+\|f^{(j)}(s)\|_{\mathbb{A}_{R+1-j}}\|f_t^{(\ell-1-j)}(s)\|_{\mathbb{A}_{R+1-(\ell-1-j)}}\\
&\mkern+120mu+\left.\|f^{(j)}(s)\|_{\mathbb{A}_{R+1-j}}\|f^{(\ell-1-j)}(s)\|_{\mathbb{A}_{R+1-(\ell-1-j)}}\right]ds
\end{split}
\end{equation*}
with $C(\alpha_o,\beta)=2\max\left\{C_1(\alpha_o,\beta),C_2(\alpha_o,\beta)\right\}$. Let us define 
\begin{equation*}
\mathcal{A}_\ell(t)=C(\alpha_o,\beta) e^{-\frac{\ell+1}{1+\ell R^2}D(R+1)}\left[\|f^{(\ell)}(t)\|_{\mathbb{A}_{R+1-\ell}}+\|f_t^{(\ell)}(t)\|_{\mathbb{A}_{R+1-\ell}}\right].
\end{equation*}
First we observe that, since $R>1$, for $0\leq j\leq \ell-1$,
\begin{equation}\label{eq:4}
\frac{\ell+1}{1+\ell R^2}+2\geq\frac{(\ell-1-j+1)}{1+(\ell-1-j)R^2}+\frac{(j+1)}{1+jR^2},
\end{equation}
so that 
\begin{equation*}
e^{-\frac{l+1}{1+lR^2}D(R+1)}\leq e^2 e^{-\frac{l-1-j+1}{1+(l-1-j)R^2}D(R+1)}e^{-\frac{j+1}{1+jR^2}D(R+1)}.
\end{equation*}
Hence, the former recursion for $\|f^{(l)}(t)\|_{X_{R+1-l}}+\|f_t^{(l)}(t)\|_{X_{R+1-l}}$ can be equivalently written as
\begin{equation*}
\mathcal{A}_\ell(t)\leq e^2 \int_{0}^{t}\sum_{j=0}^{\ell-1}\mathcal{A}_{j}(s)\mathcal{A}_{\ell-1-j}(s)ds.
\end{equation*}
We define now
$$
\mathcal{B}_\ell(t)=e^2\mathcal{A}_\ell(t),
$$
and find that $\mathcal{B}_\ell$ satisfies
\begin{equation*}
\mathcal{B}_\ell(t)\leq  \int_{0}^{t}\sum_{j=0}^{\ell-1}\mathcal{B}_{j}(s)\mathcal{B}_{\ell-1-j}(s)ds.
\end{equation*}
First, let us observe that 
\begin{equation*}
\begin{split}
\mathcal{B}_0(t)&=e^2C(\alpha_o,\beta)e^{-D(R+1)}\left[\|f^{(0)}(t)\|_{\mathbb{A}_{R+1}}+\|f_t^{(0)}(t)\|_{\mathbb{A}_{R+1}}\right]\\
&\leq \frac{e^2}{\sqrt{2\pi}\lambda}C(\alpha_o,\beta)e^{-D(R+1)} \int_{-D}^{D}e^{(R+1)|k|}\left(|\hat{f}_0|+\left|\frac{\hat{f}_1}{r^+-r^-}\right|\right)dk\\
&\leq \frac{e^2}{\sqrt{2\pi}\lambda}C(\alpha_o,\beta) \int_{-D}^{D}\left(|\hat{f}_0|+\left|\frac{\hat{f}_1}{\sqrt{|k|}}\right|\right)dk\\
&\leq \frac{C(\|f_0\|_{\mathbb{A}_0},\|f_1\|_{\mathbb{A}_0},\|f_1\|_{L^1})}{\lambda}.
\end{split}
\end{equation*}
We fix
$$
\lambda=C(\|f_0\|_{\mathbb{A}_0},\|f_1\|_{\mathbb{A}_0},\|f_1\|_{L^1})
$$
then, we prove by induction that 
\begin{equation}\label{eq:5}
\mathcal{B}_\ell(t)\leq \mathcal{C}_\ell t^{\ell}
\end{equation}
with $\mathcal{C}_l$ being the Catalan numbers, 
\begin{equation*}
\mathcal{C}_l=\sum_{j=0}^{l-1}\mathcal{C}_j\mathcal{C}_{l-1-j},
\end{equation*}
which behave as
\begin{equation}\label{eq:6}
\mathcal{C}_l\sim O(l^{-\frac32}4^l)\ \ \text{for }l>>1.
\end{equation}
Since we have already seen that \eqref{eq:5} holds for $\ell=0$ we continue with the induction step. For $\ell\geq1$ we have
\begin{equation*}
\begin{split}
\mathcal{B}_\ell&\leq \int_{0}^{t}\sum_{j=0}^{\ell-1}\mathcal{B}_{j}(s)\mathcal{B}_{\ell-1-j}(s)ds\\
&\leq \int_{0}^{t}\sum_{j=0}^{\ell-1}\mathcal{C}_js^j\mathcal{C}_{\ell-1-j}s^{l-1-j}ds\\
&=\mathcal{C}_\ell\int_{0}^{t}s^{\ell-1}ds\\
&=\mathcal{C}_\ell\frac{t^\ell}{\ell}.
\end{split}
\end{equation*}
Therefore, because of \eqref{eq:6}, we find
\begin{equation*}
\begin{split}
\|f^{(\ell)}(t)\|_{\mathbb{A}_{1}}+\|f_t^{(\ell)}(t)\|_{\mathbb{A}_{1}}&\leq\|f^{(\ell)}(t)\|_{\mathbb{A}_{R+1-\ell}}+\|f_t^{(\ell)}(t)\|_{\mathbb{A}_{R+1-\ell}}\\
&\leq [C(\alpha_o,\beta)]^{-1} e^{-2}e^{\frac{\ell+1}{1+\ell R^2}D(R+1)}4^\ell t^{\ell}
\end{split}
\end{equation*}
In a similar way 
\begin{equation*}
\|f^{(0)}(t)\|_{\mathbb{A}_{1}}+\|f_t^{(0)}(t)\|_{\mathbb{A}_{1}}\leq C(\|f_0\|_{\mathbb{A}_0},\|f_1\|_{\mathbb{A}_0},\|f_1\|_{L^1})
\end{equation*}
Then, the truncated series \eqref{eq:truncseries} satisfies the estimates
\begin{equation*}
\begin{split}
\|S_R^1\|_{\mathbb{A}_1}\leq& C(\|f_0\|_{\mathbb{A}_0},\|f_1\|_{\mathbb{A}_0},\|f_1\|_{L^1})+2\frac{\lambda}{e^2C(\alpha_o,\beta)}\sum_{\ell=1}^{R}e^{\frac{D(R+1)}{1+\ell R^2}}\left(4 e^{\frac{D(R+1)}{1+\ell R^2}}\lambda t\right)^\ell\\
\leq& C(\|f_0\|_{\mathbb{A}_0},\|f_1\|_{\mathbb{A}_0},\|f_1\|_{L^1})+2\frac{\lambda}{C(\alpha_o,\beta)}\sum_{\ell=1}^{R}\left(4 e\lambda t\right)^\ell,
\end{split}
\end{equation*}
as well as 
\begin{equation*}
\begin{split}
\|S_R^2\|_{\mathbb{A}_1}\leq& C(\|f_0\|_{\mathbb{A}_0},\|f_1\|_{\mathbb{A}_0},\|f_1\|_{L^1})+2\frac{\lambda}{C(\alpha_o,\beta)}\sum_{\ell=1}^{R}\left(4 e\lambda t\right)^\ell,
\end{split}
\end{equation*}
Thus, we conclude that, if 
\begin{equation*}
 t\leq T^*<\frac{1}{4 eC(\|f_0\|_{\mathbb{A}_0},\|f_1\|_{\mathbb{A}_0},\|f_1\|_{L^1})},
\end{equation*}
we can take the limit as $R\to+\infty$ in \eqref{eq:truncseries} and we obtain the existence of 
\begin{equation*}
f(x_1,t)=S_{\infty}^1\quad\text{and}\quad f_t(x_1,t)=S_{\infty}^2.
\end{equation*}
We also observe that the above estimates ensure $\displaystyle f, f_t\in L^{\infty}\left(0,T^*;\mathbb{A}_1\right)$. In addition, since both $f$ and $f_t$ are analytic functions in space, using the Cauchy product of power series we also find that $\displaystyle f, f_t\in C\left(0,T^*; \mathbb{A}_{0.5}\right)$.
\end{proof}

\subsection{Well-posedness for Sobolev initial data}
We recall the definition of the standard $L^2$-based Sobolev spaces
$$
H^s(\bR)=\left\{h\in L^2(\bR) \text{ s.t. }\|h\|_{H^s}^2=\int_{\bR}(1+|n|^{2s})|\hat{h}(n)|^2dn<\infty\right\}.
$$

In this section we prove the well-posedness of equation \eqref{eq:final_eq2} with periodic boundary conditions. In order to do that we will make extensive use of the following commutator estimate \cite{dawson2008decay}
\begin{align}\label{commutatorH}
\norm{ \partial_x^\ell \comm{\cH}{u}\partial_x^m v}_{L^p}\leq C\norm{ \partial_x^{\ell+m}u}_{L^\infty}\|v\|_{L^p}, && p\in(1,\infty), && \ell,m\in\mathbb{N},
\end{align}
and the fractional Leibniz rule (see \cite{grafakos2014kato,kato1988commutator,kenig1993well}):
$$
\|\Lambda^s(uv)\|_{L^p}\leq C\left(\|\Lambda^s u\|_{L^{p_1}}\|v\|_{L^{p_2}}+\|\Lambda^s v\|_{L^{p_3}}\|u\|_{L^{p_4}}\right),
$$
which holds whenever
$$
\frac{1}{p}=\frac{1}{p_1}+\frac{1}{p_2}=\frac{1}{p_3}+\frac{1}{p_4}\qquad \mbox{where $1/2<p<\infty,1<p_i\leq\infty$},
$$
and $s>\max\{0,1/p-1\}$. 

Our result reads as follows
\begin{theorem}\label{teo2}Let $\beta>0$ be a constant and $(f_0,f_1)\in H^{4.5}(\bR)\times H^{3}(\bR) $ be the initial data for equation \eqref{eq:final_eq2}. Then there exists $0<T^*$ and a unique solution
$$
(f,f_t)\in L^\infty(0,T^*,H^{4.5}(\bR))\times L^\infty(0,T^*,H^{3}(\bR)).
$$
\end{theorem}
\begin{proof}
Without loss of generality we fix $\alpha_o=\varepsilon=1$. The proof follows from appropriate energy estimates after a standard approximation using mollifiers (see \cite{granero2017model}). As a consequence, we will focus in obtaining the \emph{a priori} estimates. We define the energy
$$
\mathcal{E}(t)=\beta\|f(t)\|_{H^{4.5}}+\|f(t)\|_{H^{3.5}}+\|f_t(t)\|_{H^{3}}.
$$

In order to estimate the low order terms we test the equation against $f_t$. Integrating by parts and using that
$$
(\cH f_t)_{x_1}=\Lambda f_t,
$$
we find that
\begin{align*}
\frac{1}{2}\frac{d}{dt}\left(\|f_t\|_{L^2}^2+\|f\|_{H^{0.5}}^2+\beta\|f\|_{H^{1.5}}^2\right)&=\int_{\mathbb{R}}\left(\comm{\cH}{f}\Lambda f\right)_{x_1}f_t dx_1\leq \mathcal{E}(t)^3.
\end{align*}

Using the Fundamental Theorem of Calculus, we obtain that
$$
\frac{d}{dt}\|f\|_{L^2}^2=2\int_\bR f f_t dx_1\leq 2\|f\|_{L^2}\|f_t\|_{L^2}\leq C\mathcal{E}(t)^2
$$

To bound the high order terms we test the equation against $\Lambda^6 f_t$. Then we find
\begin{align*}
\frac{1}{2}\frac{d}{dt}\left(\|f_t\|_{H^{3}}^2+\|f\|_{H^{3.5}}^2+\beta\|f\|_{H^{4.5}}^2\right)&=I_1+I_2,
\end{align*}
where
\begin{align*}
I_1&=-\int_{\mathbb{R}}\cH\left((\cH f_t)^2\right)_{x_1}\Lambda^6 f_t dx_1,\\
I_2&=\int_{\mathbb{R}}\left(\comm{\cH}{f}\Lambda f\right)_{x_1}\Lambda^6 f_t dx_1.
\end{align*}
We integrate by parts and find that
\begin{align*}
I_1&=-\int_{\mathbb{R}}(\cH f_t)^2\Lambda^7 f_t dx_1\\
&=\int_{\mathbb{R}}(\cH f_t)^2\partial_{x_1}^6 \Lambda f_t dx_1\\
&=-\int_{\mathbb{R}}\partial_{x_1}^3(\cH f_t)^2\partial_{x_1}^3 \Lambda f_t dx_1\\
&=-2\int_{\mathbb{R}}(\cH f_t \Lambda f_{x_1x_1t}+3\Lambda f_t\Lambda f_{tx_1})\partial_{x_1}^3 \Lambda f_t dx_1\\
&= J_1^1+J_1^2,
\end{align*}
with
\begin{align*}
J_1^1&=-2\int_{\mathbb{R}}\cH f_t \Lambda f_{x_1x_1t}\partial_{x_1} \Lambda f_{x_1x_1t} dx_1,\\
J_1^2&=-6\int_{\mathbb{R}}\Lambda f_t\Lambda f_{tx_1}\partial_{x_1} \Lambda f_{x_1x_1t} dx_1.
\end{align*}
Integrating by parts and using H\"older's inequality we find that
\begin{align*}
J_1^1&=\int_{\mathbb{R}}\Lambda f_t (\Lambda f_{x_1x_1t})^2 dx_1\leq C\mathcal{E}(t)^3,\\
J_1^2&=6\int_{\mathbb{R}}\partial_{x_1}(\Lambda f_t\Lambda f_{tx_1}) \Lambda f_{x_1x_1t} dx_1\\
&=6\int_{\mathbb{R}}(\Lambda f_{x_1t}\Lambda f_{tx_1}+\Lambda f_{t}\Lambda f_{tx_1x_1}) \Lambda f_{x_1x_1t} dx_1,\\
&\leq C\mathcal{E}(t)^3.
\end{align*}
We have to handle the second nonlinear contribution. We compute that
\begin{align*}
I_2&=-\int_{\mathbb{R}}\left(\Lambda(f\Lambda f)+(ff_{x_1})_{x_1}\right)\partial_{x_1}^6 f_t dx_1\\
&=-\int_{\mathbb{R}}\partial_{x_1}^3\left(\Lambda(f\Lambda f)+(ff_{x_1})_{x_1}\right) f_{tx_1x_1x_1} dx_1\\
&=J_2^1+J_2^2+J_2^3+J_2^4,
\end{align*}
with
\begin{align*}
J_2^1&=-\int_{\mathbb{R}}\left(\Lambda(f\Lambda f_{x_1 x_1 x_1})+(ff_{x_1 x_1 x_1 x_1})_{x_1}\right) f_{tx_1x_1x_1} dx_1,\\
J_2^2&=-\int_{\mathbb{R}}\left(\Lambda(f_{x_1}\Lambda f_{x_1x_1})+(f_{x_1}f_{x_1x_1x_1})_{x_1}\right) f_{tx_1x_1x_1} dx_1\\
J_2^3&=-\int_{\mathbb{R}}\left(\Lambda(f_{x_1x_1}\Lambda f_{x_1})+(f_{x_1x_1}f_{x_1x_1})_{x_1}\right) f_{tx_1x_1x_1} dx_1\\
J_2^4&=-\int_{\mathbb{R}}\left(\Lambda(f_{x_1x_1x_1}\Lambda f)+(f_{x_1x_1x_1}f_{x_1})_{x_1}\right) f_{tx_1x_1x_1} dx_1.
\end{align*}
Using H\"older's inequality and the fractional Leibniz rule we find that
$$
J_2^2+J_2^3+J_2^4\leq C\mathcal{E}(t)^3.
$$
We observe that we can find a commutator structure in $J^1_2$. Indeed, we have that
\begin{align*}
J_2^1&=-\int_{\mathbb{R}}\left(\comm{\Lambda}{f}\Lambda f_{x_1 x_1 x_1}+f_{x_1}f_{x_1 x_1 x_1 x_1}\right) f_{tx_1x_1x_1} dx_1\\
&\leq \|f_t\|_{H^3}\|\comm{\Lambda}{f}\Lambda f_{x_1 x_1 x_1}\|_{L^2}+\mathcal{E}(t)^3.
\end{align*}
This commutator in Fourier variables takes the following form
$$
\widehat{\comm{\Lambda}{f}g}=\int_{\bR}(|n|-|n-m|)\hat{f}(m)\hat{g}(n-m)dm.
$$
In particular, using Young's inequality for convolution, Sobolev inequality and Plancherel theorem, we conclude that
$$
\|\widehat{\comm{\Lambda}{f}g}\|_{L^2}\leq \||\cdot|\hat{f}\|_{L^1}\|g\|_{L^2}\leq C\|f\|_{H^2}\|g\|_{L^2}.
$$
Inserting this commutator estimate in $J^1_2$ we conclude that
\begin{align*}
J_2^1&\leq C\|f_t\|_{H^3}\|f\|_{H^2}\|f\|_{H^4}+\mathcal{E}(t)^3\leq C\mathcal{E}(t)^3.
\end{align*}
As a consequence, we find the following differential inequality for the energy
$$
\frac{d}{dt}\mathcal{E}(t)\leq C\mathcal{E}(t)^2+\mathcal{E}(t),
$$
and we can ensure a uniform time of existence $T^*$ such that
$$
\mathcal{E}(t)\leq 2\mathcal{E}(0).
$$
Once this uniform time of existence has been obtained, the rest of the proof is standard so we only give a sketch of the argument. First we define approximate problems using mollifiers. These mollifiers are such that the previous energy estimates also holds for the regularized PDE. Then we repeat the previous computations and find the uniform time of existence $T^*$ for the sequence of regularized problems. Finally, we can pass to the limit. The uniqueness follows from a contradiction argument and we omit it.
\end{proof}

\section{The unidirectional asymptotic model for waves with odd viscosity} \label{ref:sec4}
\subsection{Derivation}
Let us consider the following 'far-field' variables,
$$
\chi=x-t,\quad \tau=\varepsilon t.
$$
An application of the chain rule leads to
$$
\frac{\partial^2}{\partial t^2}f(\chi(x,t),\tau(t))= -f_{\chi\chi}\frac{\partial \chi}{\partial t}-f_{\chi\tau}\frac{\partial \tau}{\partial t}+\varepsilon f_{\tau\chi}\frac{\partial \chi}{\partial t}+\varepsilon f_{\tau\tau}\frac{\partial \tau}{\partial t}=f_{\chi\chi}-\varepsilon f_{\chi\tau}-\varepsilon f_{\tau\chi}+\varepsilon^2 f_{\tau\tau}.
$$
So that, neglecting terms of order $\cO(\varepsilon^2)$, equation \eqref{eq:final_eq} reads
\begin{equation*}
\begin{split}
\left(f_{\chi}-2\varepsilon f_{\tau}\right)_{\chi}=\Big(&-\cH f-\varepsilon\cH\left[(\Lambda f)^2\right]+\alpha_o\cH\Lambda^2 f+\alpha_o\varepsilon\Lambda f_{\tau}\\
&+\varepsilon\comm{\cH}{f}\Lambda f-\varepsilon\alpha_o\comm{\cH}{f}\Lambda^3f-\beta\cH\Lambda^2f+\varepsilon\beta\comm{\cH}{f}\Lambda^3 f\Big)_{\chi}\ \ \text{ on }\Gamma\times[0,T].
\end{split}
\end{equation*}
Integrating on $\chi$, reordering terms and abusing notation by using $x_1$ and $t$ as variables again, we are lead to the equation
\begin{equation}\label{eq:final_eq3}
\begin{split}
2f_{t}+\alpha_o\Lambda f_{t}=&\frac{1}{\varepsilon}\left\{f_{x_1}+\cH f+(\alpha_o-\beta)\cH f_{x_1x_1}\right\}\\
&+\cH\left(\Lambda f\right)^2-\comm{\cH}{f}\Lambda f+(\alpha_o-\beta)\comm{\cH}{f}\Lambda^3 f,\ \ \text{ on }\Gamma\times[0,T].
\end{split}
\end{equation}

This equation reminds the classical Benjamin-Ono equation (cf. \cite{benjamin1967internal,ono1975algebraic}) and Burgers-Hilbert equation (cf. \cite{biello2010nonlinear}) (see also \cite{riano2021well}). It is also similar to the equation derived in \cite{duran2020asymptotic} This similarity is not only due to the linear operators. Indeed, in the new variable
$$
u=\Lambda f,
$$
we find that \eqref{eq:final_eq3} contains the classical Burgers term:
\begin{equation}\label{eq:final_eq32}
\begin{split}
u_{t}=&\frac{1}{\varepsilon}\frac{1}{2+\alpha_o\Lambda}\left\{u_{x_1}+\cH u+(\alpha_o-\beta)\cH u_{x_1x_1}\right\}\\
&+\frac{1}{2+\alpha_o\Lambda}\left\{-\partial_{x_1}\left(u\right)^2-\Lambda\comm{\cH}{f}u+(\alpha_o-\beta)\Lambda\comm{\cH}{f}\Lambda^2 u\right\},\ \ \text{ on }\Gamma\times[0,T].
\end{split}
\end{equation}

\subsection{Well-posedness for Sobolev initial data}
In this section we study the well-posedness of equation \eqref{eq:final_eq3}.

\begin{theorem}\label{teo3}Let $\alpha_o>0$ and $\beta\geq0$ be two constants and $f_0\in H^{3}(\bR)$ be the initial data for equation \eqref{eq:final_eq3}. Then there exists $0<T^*$ and a unique solution
$$
f\in L^\infty(0,T^*,H^{3}(\bR)).
$$
\end{theorem}
\begin{proof}
As before, the proof follows from apropriate energy estimates and a suitable sequence of approximate problems. Thus, we start with the energy estimates. We test the equation against $\Lambda^6 f$ and we find that
$$
\frac{1}{2}\frac{d}{dt}\|f\|_{H^3}^2=\int_\bR\frac{1}{2+\alpha_o\Lambda}\left\{\cH\left(\Lambda f\right)^2-\comm{\cH}{f}\Lambda f+(\alpha_o-\beta)\comm{\cH}{f}\Lambda^3 f\right\}\Lambda^6 f dx_1.
$$
Integrating by parts we have that
\begin{align*}
\frac{1}{2}\frac{d}{dt}\|f\|_{H^3}^2&=-\int_\bR\frac{1}{2+\alpha_o\Lambda}\left\{\cH\left(\Lambda f\right)^2-\comm{\cH}{f}\Lambda f+(\alpha_o-\beta)\comm{\cH}{f}\Lambda^3 f\right\}\partial_{x_1}^2 \Lambda^4f dx_1\\
&=-\int_\bR\partial_{x_1}^2\left\{\cH\left(\Lambda f\right)^2-\comm{\cH}{f}\Lambda f+(\alpha_o-\beta)\comm{\cH}{f}\Lambda^3 f\right\} \frac{\Lambda^4}{2+\alpha_o\Lambda}f dx_1.
\end{align*}
Using the bound
$$
\left\|\frac{\Lambda^4}{2+\alpha_o\Lambda}f \right\|_{L^2}\leq \|f\|_{H^3},
$$
together with the commutator estimate for the Hilbert transform \eqref{commutatorH}, we find that
\begin{align*}
\frac{1}{2}\frac{d}{dt}\|f\|_{H^3}^2&\leq C\|f\|_{H^3}^3.
\end{align*}
This differential inequality leads to a uniform time of existence. Equipped with the uniform time of existence, the existence of solution can be obtained using a sequence of regularized problems (see \cite{granero2017model}). The uniqueness is a consequence of a contradiction argument and the regularity of the solution.
\end{proof}

\subsection{Distributional solution with limited regularity} Let us consider the case $\alpha_o>0$ and $\beta=\alpha_o$. 

%\textcolor{red}{Furthermore, in this section we are going to consider the equation \eqref{eq:final_eq3_intro} when $\Gamma=[-\pi,\pi]$ with periodic boundary conditions. This assumption on the domain will simplify the functional analytic setting required in our theorem.}

Before stating our result, we define our concept of distributional solution: we say that $f$ is a distributional solution of \eqref{eq:final_eq3_intro} if and only if 
\begin{multline*}
-\int_{\bR} (2+\alpha_o\Lambda)\varphi(x_1,0) f_0(x_1)dx_1ds-\int_0^T\int_{\bR} (2+\alpha_o\Lambda)\varphi_t(x_1,s) f(x_1,s)dx_1ds\\
=-\frac{1}{\varepsilon}\int_0^T\int_{\bR} f(x_1,s)\varphi_{x_1}(x_1,s)+\cH \varphi(x_1,s) f(x_1,s)dx_1ds\\
+\int_0^T\int_{\bR} \left\{\cH\left(\Lambda f\right)^2-\comm{\cH}{f}\Lambda f\right\}\varphi(x_1,s)dx_1ds,
\end{multline*}
for all $\varphi\in C^\infty_c([0,T)\times\bR)$.

\begin{theorem}\label{teo4}Let $\alpha_o>0$ and $\beta=\alpha_o$ be two constants and $f_0\in H^{1.5}(\bR)$ be the initial data for equation \eqref{eq:final_eq3}. Then there exists $0<T^*$ and at least one distributional solution
$$
f\in L^\infty(0,T^*,H^{1.5}(\bR)).
$$
\end{theorem}
\begin{proof}
We consider the regularized problem
\begin{equation}\label{eq:final_eq3_distri}
\begin{split}
2f^{(n)}_{t}+\alpha_o\Lambda f^{(n)}_{t}-\frac{1}{n}f^{(n)}_{x_1 x_1}=&\frac{1}{\varepsilon}\left\{f^{(n)}_{x_1}+\cH f^{(n)}\right\}\\
&+\cH\left(\Lambda f^{(n)}\right)^2-\comm{\cH}{f^{(n)}}\Lambda f^{(n)},\ \ \text{ on }\Gamma\times[0,T],
\end{split}
\end{equation}
with the mollified initial data
$$
f^{(n)}(x_1,0)=\rho_n*f_0(x_1),
$$
where $\rho_n$ is a standard Friedrich mollifier.

We test the equation against $f^{(n)}$ and use H\"older and Sobolev inequalities to find
$$
\frac{d}{dt}\|f^{(n)}\|_{L^2}^2+\frac{d}{dt}\|\Lambda^{0.5}f^{(n)}\|_{L^2}^2+\frac{1}{n}\|f^{(n)}_{x_1}\|_{L^2}^2\leq C\|f^{(n)}_{x_1}\|_{L^2}^3.
$$
Now we test the equation against $\Lambda^{2}f^{(n)}$. We find that
\begin{align*}
\frac{d}{dt}\|\Lambda f^{(n)}\|_{L^2}^2+\frac{d}{dt}\|\Lambda^{1.5}f^{(n)}\|_{L^2}^2+\frac{1}{n}\|f^{(n)}_{x_1 x_1}\|_{L^2}^2&=\int_{\bR}\cH\left(\Lambda f^{(n)}\right)^2\Lambda^{2}f^{(n)}dx_1\\
&\quad-\int_{\bR}\comm{\cH}{f^{(n)}}\Lambda f^{(n)}\Lambda^{2}f^{(n)}dx_1.
\end{align*}
The first nonlinear contribution vanishes. Indeed,
\begin{align*}
\int_\bR\cH\left(\Lambda f^{(n)}\right)^2\Lambda^{2}f^{(n)}dx_1&=\int_{\bR}\left(\Lambda f^{(n)}\right)^2\Lambda f^{(n)}_{x_1}dx_1\\
&=0.
\end{align*}
The second nonlinear contribution can be handled as follows
\begin{align*}
-\int_{\bR}\comm{\cH}{f^{(n)}}\Lambda f^{(n)}\Lambda^{2}f^{(n)}dx_1&=-\int_{\bR}(\cH(f^{(n)}\Lambda f^{(n)})+f^{(n)}f^{(n)}_{x_1}))\Lambda^{2}f^{(n)}dx_1\\
&=\int_{\bR}(\cH(f^{(n)}\Lambda f^{(n)})+f^{(n)}f^{(n)}_{x_1}))f^{(n)}_{x_1 x_1}dx_1\\
&\leq C\|f^{(n)}_{x_1}\|_{L^3}^3\\
&\leq C\|f^{(n)}_{x_1}\|_{H^{1/6}}^3\\
&\leq C\|f^{(n)}_{x_1}\|_{L^2}^2\|f^{(n)}_{x_1}\|_{H^{0.5}}.
\end{align*}
As a consequence, if we define
$$
\mathcal{E}(t)=\|f^{(n)}\|_{L^2}^2+\|\Lambda^{0.5}f^{(n)}\|_{L^2}^2+\|\Lambda f^{(n)}\|_{L^2}^2+\|\Lambda^{1.5}f^{(n)}\|_{L^2}^2,
$$
we have that
$$
\frac{d}{dt}\mathcal{E}(t)\leq C \mathcal{E}(t)^2.
$$
We thus conclude the uniform-in-$n$ time of existence $T^*$ such that
$$
f^{(n)}\in L^\infty(0,T^*,H^{1.5}(\bR)),
$$
with a bound that is independent of $n$. This implies that
$$
f^{(n)}\overset{\ast}{\rightharpoonup} f\in L^\infty(0,T^*,H^{1.5}(\bR)).
$$

Furthermore, using the regularity of $f$ together with
$$
\|(\Lambda f)^2\|_{L^2}^2= \|\Lambda f\|_{L^4}^4\leq C\|\Lambda f\|_{H^{0.25}}^4\leq C\|\Lambda f\|_{L^2}^2\|\Lambda f\|_{H^{0.5}}^2,
$$
and we can compute
$$
f^{(n)}_t\in L^\infty(0,T^*,L^2(\bR)),
$$
with a bound that is independent of $n$.

In particular, 
$$
f^{(n)}\in L^\infty(0,T^*,H^{1.5}([-1,1])),
$$
$$
f^{(n)}_t\in L^\infty(0,T^*,L^2([-1,1])).
$$
Then, a standard application of the Aubin-Lions Theorem ensures that we can obtain a subsequence such that
$$
f^{(n_1(j))}\rightarrow f_{(1)}\in L^2(0,T^*,H^{1}([-1,1])).
$$
Similarly, the elements in this sequence satisfy
$$
f^{(n)}\in L^\infty(0,T^*,H^{1.5}([-2,2])),
$$
$$
f^{(n)}_t\in L^\infty(0,T^*,L^2([-2,2])),
$$
so, we can extract another subsequence such that
$$
f^{(n_2(j))}\rightarrow f_{(2)}\in L^2(0,T^*,H^{1}([-2,2])).
$$
Due to the uniqueness of the limit, we have that
$$
f_{(1)}=f_{(2)}
$$
at least in the common interval $[-1,1]$. Then, for each $m$, we can repeat this procedure and find different subsequences $f^{(n_m(j))}$ of the original sequence $f^{(n)}$. Now we use Cantor's diagonal argument. Then, we define the sequence
$$
f^{(\ell)}=f^{(n_\ell(\ell))}
$$
These $f^{(\ell)}$ are a subsequence of the original sequence $f^{(n)}$ and then 
$$
f^{(\ell)}\in L^\infty(0,T^*,H^{1.5}(\bR)),
$$
$$
f^{(\ell)}_t\in L^\infty(0,T^*,L^2(\bR)).
$$
Now, we observe that for each interval $[-k,k]$, we have that
$$
f^{(\ell)}\rightarrow f\in  L^2(0,T^*,H^{1}([-k,k])).
$$
Indeed, it is enough to note that the elements $f^{(\ell)}$ for $\ell\geq k+1$ are elements of a subsequence that converges in $[-k,k]$ and that the resulting limit must be unique. In addition, if we fix an arbitrary compact set $\mathcal{U}\subset\bR$, we have that
$$
f^{(\ell)}\rightarrow f\in  L^2(0,T^*,H^{1}(\mathcal{U})).
$$
Now, if we fix a text function $\varphi$, we have that the distributional form of the approximate problems reads
\begin{multline*}
-\int_{\bR} (2+\alpha_o\Lambda)\varphi(x_1,0) \rho_\ell*f_0(x_1)dx_1ds-\int_0^T\int_{\bR} (2+\alpha_o\Lambda)\varphi_t(x_1,s) f^{(\ell)}(x_1,s)dx_1ds \\
=-\frac{1}{\varepsilon}\int_0^T\int_{\bR} f^{(\ell)}(x_1,s)\varphi_{x_1}(x_1,s)+\cH \varphi(x_1,s) f^{(\ell)}(x_1,s)dx_1ds\\
+\int_0^T\int_{\bR} \left\{\cH\left(\Lambda f^{(\ell)}\right)^2-\comm{\cH}{f^{(\ell)}}\Lambda f^{(\ell)}\right\}\varphi(x_1,s)dx_1ds+\int_0^T\int_{\bR}\frac{f^{(\ell)}}{n(\ell)}\varphi_{x_1 x_1} dx_1ds.
\end{multline*}
Due to weak-$*$ convergence it is easy to see that the linear terms converge 
\begin{multline*}
-\int_{\bR} (2+\alpha_o\Lambda)\varphi(x_1,0) \rho_\ell*f_0(x_1)dx_1ds-\int_0^T\int_{\bR} (2+\alpha_o\Lambda)\varphi_t(x_1,s) f^{(\ell)}(x_1,s)dx_1ds\\
\rightarrow -\int_{\bR} (2+\alpha_o\Lambda)\varphi(x_1,0) f_0(x_1)dx_1ds-\int_0^T\int_{\bR} (2+\alpha_o\Lambda)\varphi_t(x_1,s) f(x_1,s)
dx_1ds, 
\end{multline*}
\begin{multline*}
-\frac{1}{\varepsilon}\int_0^T\int_{\bR} f^{(\ell)}(x_1,s)\varphi_{x_1}(x_1,s)+\cH \varphi(x_1,s) f^{(\ell)}(x_1,s)dx_1ds\\
\rightarrow -\frac{1}{\varepsilon}\int_0^T\int_{\bR} f(x_1,s)\varphi_{x_1}(x_1,s)+\cH \varphi(x_1,s) f(x_1,s)dx_1ds,
\end{multline*}
$$
\int_0^T\int_{\bR}\frac{f^{(\ell)}}{n(\ell)}\varphi_{x_1 x_1} dx_1ds\rightarrow0.
$$
The first nonlinear contribution can be handled as follows:
\begin{align*}
I&=\int_0^T\int_{\bR} \cH\left[\left(\Lambda f^{(\ell)}+\Lambda f\right)\left(\Lambda f^{(\ell)}-\Lambda f\right)\right]\varphi dx_1ds\\
&=\int_0^T\int_{-M}^M \left(\Lambda f^{(\ell)}+\Lambda f\right)\left(\Lambda f^{(\ell)}-\Lambda f\right)\cH\varphi dx_1ds\\
&\leq C_\varphi\|f^{(\ell)}\|_{H^1(\bR)}\|f^{(\ell)}-f\|_{H^1([-M,M])}\rightarrow0.
\end{align*}
The commutator term is of lower order and can be handled similarly. Then, passing to the limit we conclude that the limit function $f$ satisfies the distributional form \eqref{eq:final_eq3_distri}.
\end{proof}

\begin{remark}
We would like to emphasize that it is possible to find uniform-in-time energy estimates for the $H^1$ norm (instead of the $H^{1.5}$ norm). However,  the notion of solution seems unclear at that level of regularity.
\end{remark}

\subsection{Numerical study}
In this section we report a preliminary numerical study of equation \eqref{eq:final_eq32} with periodic boundary conditions in $[-\pi,\pi]$. 

These numerical results have been obtained using a spectral method to simulate both the differential and the singular integral operators. In particular, in order to simulate \eqref{eq:final_eq32} we use the Fourier-collocation method. This method considers a discretization of the spatial domain with $N$ uniformly distributed points. Then we use the Fast Fourier Transform and Inverse Fast Fourier Transform (IFFT) routines already implemented in Octave to jump between the physical and the frequency spaces. In this way we can take advantage of the fact that, in Fourier variables, the differential operators and the Hilbert transform are defined by multipliers. With this method, the problem reduces to a system of ODEs in Fourier space. To advance in time we used the standard adaptative Runge-Kutta scheme implemented in the Octave function \textit{ode45}. 
\begin{figure}[h!]
    \centering
    \includegraphics[scale=0.25]{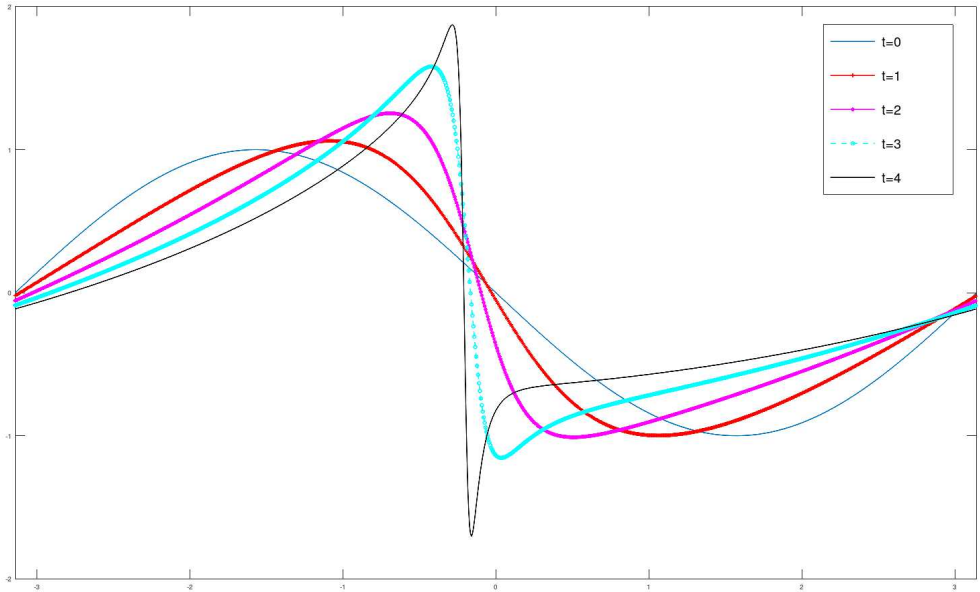}
\caption{Evolution in the case $\epsilon=\alpha=\beta=1$ and $N=2^{10}$.}
\label{fig1}
\end{figure}
\begin{figure}[h!]
    \centering
    \includegraphics[scale=0.35]{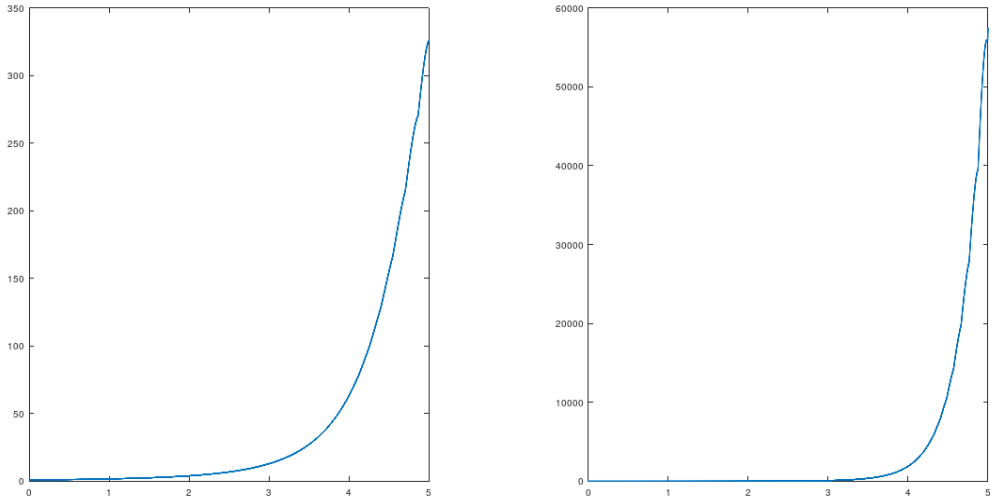}
\caption{$\|u_{x_1}(t)\|_{L^\infty}$ (left) and $\|u_{x_1 x_1}(t)\|_{L^\infty}$ (right) in the case $\epsilon=\alpha=\beta=1$ and $N=2^{11}$.}
\label{fig2}
\end{figure}
\begin{itemize}
\item[Case 1:] We consider the initial data for \eqref{eq:final_eq32} given by 
$$
u(x_1,0)=-\sin(x_1).
$$
The physical parameters are $\epsilon=\alpha=\beta=1$. Here we see that the solution is getting steeper and steeper (see figure \ref{fig1}). However, when we compute $\|u_{x_1}(t)\|_{L^\infty}$ and $\|u_{x_1 x_1}(t)\|_{L^\infty}$ we cannot conclude the existence of a finite time singularity (see figure \ref{fig2}). In particular, both the first and the second derivative grow, however, they seem to remain $\cO(10^2)$ and $\cO(10^4)$ respectively. 

\item[Case 2:] We consider the initial data for \eqref{eq:final_eq32} given by 
$$
u(x_1,0)=-10 \sin(10 x_1).
$$
The physical parameters are $\epsilon=0.1,\alpha=\beta=1$. Here we see that the solution oscillates. In particular, the solution is getting steeper but then depletes (see figure \ref{fig3}).  
\begin{figure}[h!]
    \centering
    \includegraphics[scale=0.35]{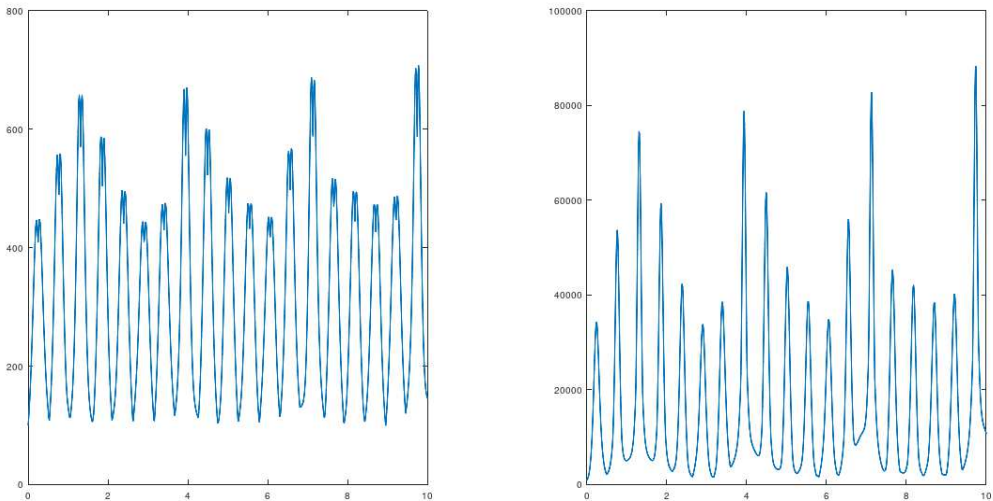}
\caption{$\|u_{x_1}(t)\|_{L^\infty}$ (left) and $\|u_{x_1 x_1}(t)\|_{L^\infty}$ (right) in the case $\epsilon=0.1,\alpha=\beta=1$ and $N=2^{11}$.}
\label{fig3}
\end{figure}
\end{itemize}

\section{Discussion} \label{ref:sec5}
In this paper we have obtained new asymptotic models for both bidirectional and unidirectional gravity-capillary odd waves. Besides the derivation, we have also studied some of their mathematical properties rigorously. In particular, we have proved a number of local in time well-posedness results in appropriate spaces. Furthermore, we have also studied the unidirectional model numerically trying to find a numerical scenario of finite time singularities. At this point, this scenario remains undetermined and the question of finite time singularities or the global existence of smooth solutions remain as open problems.

\section*{Acknowledgments}
R.G-B was supported by the project ”Mathematical Analysis of Fluids and Applications” with reference PID2019-109348GA-I00/AEI/ 10.13039/501100011033 and acronym ``MAFyA” funded by Agencia Estatal de Investigaci\'on and the Ministerio de Ciencia, Innovacion y Universidades (MICIU). Part of this research was performed when R.G-B was visiting the University Carlos III of Madrid. R.G-B is grateful to the Mathematics Department of the University Carlos III of Madrid for their hospitality during this visit. R.G-B would like to acknowledge discussions with MA. Garc\'ia-Ferrero, S. Scrobogna and D. Stan.

\bibliographystyle{plain}
%\bibliography{references}

\end{document}